\documentclass[10pt, a4paper, reqno]{amsart}

\usepackage{amsmath, amssymb, amsfonts, enumerate}
\usepackage[english]{babel}
\usepackage{hyperref}
\usepackage{mathrsfs}

\theoremstyle{plain}
\newtheorem{lemma}{Lemma}
\newtheorem{theo}{Theorem}
\newtheorem{remark}{Remark}

\newcommand{\Z}{\mathbb{Z}}
\newcommand{\Q}{\mathbb{Q}}
\newcommand{\kI}{\mathfrak{I}}
\newcommand{\kC}{\mathfrak{C}}
\newcommand{\kT}{\mathfrak{T}}
\newcommand{\kN}{\mathfrak{N}}
\newcommand{\kp}{\mathfrak{p}}

\begin{document}

\title[On the number of classes of extension]{Determination of the number of isomorphism\\
classes of extensions of a $\kp$-adic field}
\author{Maurizio Monge}
\email{maurizio.monge@sns.it}
\address{Scuola Normale Superiore di Pisa - Piazza dei Cavalieri, 7 -
  56126 Pisa}

\subjclass[2000]{11S15, 11S20, Secondary: 05A19, 20D60}
\keywords{p-adic field, local field, enumeration, isomorphism class of extensions, local class field
  theory, Krasner formula, ramification}
\date{\today}

\begin{abstract}
We deduce a formula enumerating the isomorphism classes of extensions
of a $\kp$-adic field $K$ with given ramification $e$ and inertia
$f$. The formula follows from a simple group-theoretic lemma, plus the
Krasner formula and an elementary class field theory computation. It
shows that the number of classes only depends on the ramification and
inertia of the extensions $K/\Q_p$, and $K(\zeta_{p^m})/K$ obtained
adding the $p^m$-th roots of $1$, for all $p^m$ dividing $e$.
\end{abstract}

\maketitle

Let $K$ be a finite extension of $\Q_p$ with residue field $\bar K$.
Let $n_0=[K : \Q_p]$, and let respectively $e_0 = e(K/\Q_p)$ and
$f_0=f(K/\Q_p)$ be the absolute ramification index and inertia
degree. Formulas for the total number of extensions with given degree
in a fixed algebraic closure were computed by Krasner and Serre
\cite{Krasner1962,Serre1978}, as well as formulas counting all totally
ramified extensions and totally ramified extension with given
valuation of the discriminant.

We will show how it is possible, with some help from class field
theory, to modify such formulas to enumerate the isomorphism classes
of extensions with fixed ramification and inertia, and all extension
with fixed degree. We intend to show in a forthcoming paper the
application of this method to the determination of the number classes
of totally ramified extensions with given discriminant, when possible.

The problem of enumerating isomorphism classes of $\kp$-adic field had
been solved in a special case by Hou-Keating \cite{hou2004enumeration}
for extensions with ramification $e$ satisfying $p^2 \nmid e$, with a
partial result when $p^2 \parallel e$.

\section{Group theoretic preliminaries}

For $k\geq 1$, let $C_k\cong\Z/k\Z$ denote the cyclic group of order
$k$. Given a group $G$, let $\mathscr{F}$ be a family of subgroups of
$G$ which is closed under conjugation and contain a finite number of
subgroups of fixed index in $G$. For each integer $n$, let $\kI_n$ be
the number of conjugacy classes of subgroups $H \in \mathscr{F}$
having index $n$ in $G$. For a finite group $Q$, let $\kT_n(Q)$ denote
the number of two-steps chains of subgroups
$H\vartriangleleft{}J\leq{}G$ such that $H \in \mathscr{F}$, $(G:H) =
n$ and $H$ is normal in $J$ with $J/H \cong Q$.

\begin{lemma}
  For a group $G$ and a family of subgroup $\mathscr{F}$ which is
  closed under conjugation and containing a finite number of subgroups
  of fixed index in $G$ we have
\begin{equation}
\label{eq:1:1}
  \kI_n = \frac{1}{n}\sum_{d|n} \phi(d)\kT_n(C_d),
\end{equation}
for each $n$.
\end{lemma}
\begin{proof}
For fixed $H\in \mathscr{F}$ with $(G:H)=n$, let's compute the
contribute of the chains of the form $H \vartriangleleft J \leq G$. All
the admissible $J$ are contained in the normalizer $N_G(H)$, and the
number of subgroups in $N_G(H)/H$ isomorphic to $C_d$ multiplied by
$\phi(d)$ counts the number of elements of $N_G(H)/H$ with order
precisely equal to $d$, because $\phi(d)$ is the number of possible
  generators of a group isomorphic to $C_d$. The contribute for all possible $d$ is hence
equal to $(N_G(H):H)$, and having $H$ precisely $(G:N_G(H))$
conjugates its conjugacy class contributes $n$ to the sum, i.e. $1$ to
the full expression.
\end{proof}

The absolute Galois group of a $\kp$-adic field has only a finite
number of closed subgroups with fixed index and consequently the
closed subgroups can be taken as family $\mathscr{F}$. By Galois
theory the formula \eqref{eq:1:1} can be interpreted denoting with
$\kI_n$ the number of isomorphism classes of extensions $L/K$ of
degree $n$ over a fixed field $K$, and with $\kT_n(Q)$ the number of
all towers of extensions $L/F/K$ such that $[L:K] = n$, and $L/F$ is
Galois with group isomorphic to $Q$.  Similarly, the above formula can
be used to count, say, extensions with prescribed ramification and
inertia, all the extensions of given degree, or totally ramified
extensions with prescribed valuation of the different, restricting the
computation via an appropriate choice of the family $\mathscr{F}$ of
subgroups of $G$.

The above formula can be applied to local fields with great
effectiveness because the number of cyclic extension with prescribed
ramification and inertia (which will be carried over in the next
section) has little dependence on the particular field taken into
account, and only depends on the absolute degree over
$\Q_p$, the absolute inertia, and the $p$-part of the group of the
roots of the unity.

\section{On the number of cyclic extensions}

In this section we briefly deduce via class field theory the number
$\kC(F,e,f)$ of cyclic extensions of a $\kp$-adic field $F$ with
prescribed ramification $e$, inertia $f$ and degree $d=ef$, and the
number $\kC(F,d)$ of all cyclic extension of degree $d$, for any $d$.

By class field theory (see \cite{Serre1979local,fesenko2002local}),
the maximal abelian extension with exponent $d$ has Galois group
isomorphic to the biggest quotient of $F^\times$ which has exponent
$d$, and consequently its Galois group is isomorphic to
$F^\times/(F^\times)^d$. Furthermore, the upper numbering ramification
groups are the images of the principal units $U_0,U_1,\dots$ under
this isomorphism.

The choice of a uniformizer $\pi$ provides a factorization
$F=\langle \pi \rangle \times U_0$, and as is well known \cite[see (3) of
Cor. 6.5, and also Prop. 5.4 and Cor. 7.3]{fesenko2002local}, $U_0$ is
isomorphic to the direct product of the group of the roots of the
unity $\mu_F$ and a free $\Z_p$ module with rank
$m=[F:\Q_p]$. Consequently calling $G$ the Galois group of the maximal
abelian extension with exponent $d$ we have
\begin{equation*}
  G \cong C_d \times C_z \times C_{p^r}^m \times C_{p^{\min\{\xi,r\}}},
  \qquad G^0 \cong \{1\} \times C_z \times C_{p^r}^m \times C_{p^{\min\{\xi,r\}}}
\end{equation*}
where $d = p^rk$ with $(d,k) = 1$, $z$ is the g.c.d. of $k$ and the
order $|\bar F^\times| = p^{f(F/\Q_p)}-1$ of the group of the roots of
the unity with order prime with $p$, and $\xi$ is the integer such
that $p^\xi$ is the order of the group of the roots of the unity with
$p$-power order.

The number of subgroups $H \subseteq G$ such that $G/H\cong{}C_d$ and
$G/(HG^0)\cong{}C_f$ will be computed using the duality theory of
finite abelian groups. Let $\hat{G}=Hom(G,\Q/\Z)$, and
for each subgroup $H$ of $G$ put
$H^\bot=\{\phi\in\hat{G}:\phi(x)=0,\text{ for }x \in H\} \cong
\widehat{G/H}$.

The conditions on $H$ amounts to having $H^\bot \cong C_d$, and
$(HG^0)^\bot = H^\bot \cap (G^0)^\bot \cong C_f$. Since we have
(non-canonically) that
\begin{equation*}
(G^0)^\bot \cong C_d \times \{1\} \times \{1\} \times \{1\} \subseteq C_d \times C_z \times
C_{p^r}^m \times C_{p^{\min\{\xi,r\}}} \cong \hat G,
\end{equation*}
we must count the number of subgroups isomorphic to $C_d$, generated
by an element of the form $(x,y)$ with $x \in C_d$ and $y \in C_z
\times C_{p^r}^m \times C_{p^{\min\{\xi,r\}}}$ say, such that the
intersection with $(G^0)^\bot$ is isomorphic to $C_f$. The order of $e
\cdot x$ must be equal to $f$, and since the map of multiplication by
$e$ from $C_d$ to $C_d$ is $e$-to-$1$ and has image isomorphic to
$C_f$ we have that the number of possible $x$ is $e$ times the number
of generators of $C_f$, and hence is equal to $e\phi(f)$.

The $y$ coordinate on the other hand should have order precisely equal
to $e$, which we write $e = p^sh$ with $(h,p)=1$, and this is
impossible if $h \nmid z$, while if $h\mid z$ we have
\begin{equation*}
  \phi(h) \cdot \Pi_p(m,s,\xi)
\end{equation*}
possibilities for $y$, where for all $p,m,s,\xi$ we define for
convenience
\begin{equation}
\label{eq:2:0}
  \Pi_p(m,s,\xi) = \left\{\begin{array}{cl}
 1 & \text{ if }s = 0, \\
 p^{ms+\min\{\xi,s\}}-p^{m(s-1)+\min\{\xi,s-1\}} &\text{ if }s > 0,
\end{array}
\right.
\end{equation}
which counts the number of elements of order $p^s$ in a group
isomorphic to $C_{p^r}^m \times C_{p^{\min\{\xi,r\}}}$ (for any $r\geq s$).

 Since we counted the number of good generators, to
obtain the number of good groups we must divide by $\phi(ef)$ obtaining
\begin{equation}
\label{eq:2:1}
  \kC(F,e,f) = \frac{e \phi(h)
    \phi(f)}{\phi(ef)} \cdot \Pi_p(m,s,\xi)
\end{equation}
if $h \mid (p^{f(F/\Q_p)}-1)$, while $\kC(F,e,f) = 0$ if $h \nmid
(p^{f(F/\Q_p)}-1)$.

The total number of cyclic extensions of degree $d = p^rk$ can be
deduced similarly, considering that
\begin{equation*}
  G \cong C_k \times C_z \times C_{p^r}^{m+1} \times C_{\min\{\xi,r\}}.
\end{equation*}
Let's consider the function $\psi(u,v)$ which for natural $u,v$ counts
the number of elements with order $u$ in the group $C_u \times C_v$,
and can be expressed as
\begin{equation}
\label{eq:2:2}
  \psi(u,v) = u\cdot(u,v) \cdot \prod_{\substack{\ell\text{ prime}\\\ell\mid u/(u,v)}}
      \left(1-\frac{1}{\ell}\right)
   \cdot \prod_{\substack{\ell\text{ prime}\\\ell \mid u,\ \ell\nmid
      u/(u,v)}}\left(1-\frac{1}{\ell^2}\right).
\end{equation}
A computation similar to what done above tells us that the total
number of $C_d$-extensions is
\begin{equation}
\label{eq:2:3}
  \kC(F, d) = \frac{\psi(k, p^{f(K/\Q_p)}-1)}{\phi(d)} \cdot
      \Pi_p(m+1,r,\xi).
\end{equation}

\section{The formula for the number of isomorphism classes of extensions}

Let's recall Krasner formula for the number $\kN(K,e,f)$ of extensions
with ramification $e=p^sh$ (with $(p,h)=1$) and inertia $f$ of a field
$K$ having absolute degree $n_0 = [K:\Q_p]$. The formula is
\begin{equation*}
  \kN(K,e,f) = e \cdot \sum_{i=0}^s p^i\left(p^{\varepsilon(i)N} - p^{\varepsilon(i-1)N}\right),
\end{equation*}
where $N = n_0ef$ and
\begin{equation*}
\varepsilon(i) = \left\{\begin{array}{cl}
-\infty & \text{if }i=-1, \\
0 & \text{if }i=0, \\
p^{-1}+p^{-2}+\dots+p^{-i} & \text{if }i>0.
\end{array}\right.
\end{equation*}
Denote by convenience with $\Sigma_p(N,s)$ the sum in the Krasner
formula
\begin{equation} \label{eq:3:1}
  \Sigma_p(N,s) = \sum_{i=0}^s p^i\left(p^{\varepsilon(i)N} - p^{\varepsilon(i-1)N}\right),
\end{equation}
so that it can be written as $\kN(K,e,f) = e \cdot \Sigma_p(N,s)$.

If we ignore for a moment the dependence of $\kC(F,e,f)$ on the group
of $p$-power roots of the unit, we can count the number of isomorphism
classes of extensions iterating on all the towers $L/F/K$ with $F/K$
having ramification $e'$ and inertia $f'$, and $L/F$ with ramification
$e''$ and inertia $f''$, with $e'e''=e$ and $f'f''=f$. Indicating with
$F^{e',f'}$ a ``generic'' extension with ramification $e'$ and inertia
$f'$ over $K$, we obtain
\begin{align*}
  \kI(K,e,f) &= \frac{1}{n}\sum_{\substack{f'f''=f\\e'e''=e}}
  \phi(e''f'') \cdot \kN(K,e',f') \cdot \kC(F^{e',f'},e'',f'') \\
  &= \frac{1}{f}\sum_{
    \substack{f'f''=f\\e'e''=e\\
    h'' \big| \big(p^{f_0f'}-1\big)}}
      \phi(h'')\phi(f'') \cdot \Sigma_p(n_0e'f',s') \cdot \Pi_p(n_0e'f',s'',\xi),
\end{align*}
where in all the sum we always put $e'=p^{s'}h'$, $e''=p^{s''}h''$
with $(p,h') = (p,h'') = 1$.

But unluckily the $\xi$ describing the order of the group of $p$-power
roots of the unity in our ``generic'' extension $F^{e',f'}$ is not well
defined, and depends on the particular extension $F^{e',f'}/K$. If we start
putting $\xi = 0$ while counting all the towers $L/F/K$, each factor
$\Pi_p(\cdot,\cdot,0)$ should be corrected by a term
$\Pi_p(\cdot,\cdot,1) - \Pi_p(\cdot,\cdot,0)$ for all towers with $F
\supseteq K(\zeta_p)$, another correction term $\Pi_p(\cdot,\cdot,2) -
\Pi_p(\cdot,\cdot,1)$ is required for all
$F\supseteq{}K(\zeta_{p^2})$, and so on.

Consequently, let's define for convenience the quantity
$\Delta_p(m,s,0) = \Pi_p(m,s,0)$, and
$\Delta_p(m,s,i)=\Pi_p(m,s,i)-\Pi_p(m,s,i-1)$ for $i>0$, that gives
the main contribution and all the correction terms, and can be written
explicitly as
\begin{equation} \label{eq:3:2}
\Delta_p(m,s,i) = \left\{\begin{array}{cl}
  1                       & \text{if }s=i=0, \\
  (p^m-1)p^{m(s-1)}         & \text{if }s>i=0, \\
  (p-1)(p^m-1)p^{m(s-1)+i-1} & \text{if }s>i>0, \\
  (p-1)p^{ms+s-1}           & \text{if }s=i>0, \\
  0                       & \text{if }i>s.
\end{array}\right.
\end{equation}
Adding all correctiong terms for all tower where $F$ contains all the
various extension $K(\zeta_{p^i})$ for $i\leq s$, we have

\begin{theo} \label{th1}
The number of isomorphism classes of extensions with ramification $e$
and inertia $f$ of a field $K$ of absolute degree $n_0 = [K:\Q_p]$
and absolute inertia $f_0 = f(K/\Q_p)$ is
\begin{equation*}
  \kI(K,e,f) = \frac{1}{f}\sum_{
    \substack{0\leq i \leq s\\f'f''f^{(i)}=f\\e'e''e^{(i)}=e\\
    h'' \big| \big(p^{f_0f^{(i)}f'}-1\big)}}
      \frac{\phi(h'')\phi(f'')}{e^{(i)}}
          \cdot \Sigma_p(N',s')
      \cdot \Delta_p(N',s'',i),
\end{equation*}
where $\Sigma_p$ and $\Delta_p$ are respectively defined in the
\eqref{eq:3:1} and \eqref{eq:3:2}, we have put
\begin{equation*}
  e^{(i)} = e(K(\zeta_{p^i})/K),\quad f^{(i)} =
  f(K(\zeta_{p^i})/K),\quad n^{(i)} = e^{(i)}f^{(i)},
\end{equation*}
and where throughout the sum (for all $i$, $e'$, $f'$, etc) we have put
\begin{equation*}
N' = n_0n^{(i)}e'f'
\end{equation*}
and
\begin{equation*}
  e=p^sh,\quad e'=p^{s'}h', \quad e''=p^{s''}h''
\end{equation*}
with $h,h',h''$ all prime with $p$.
\end{theo}

\begin{remark}
When $s=0$ (i.e. $p\nmid e$) the formula has the much simpler form
\begin{align*}
\kI(K,e,f) &= \frac{1}{f}\sum_{
    \substack{f'f''=f\\e'e''=e\\
    e'' \big| \big(p^{f_0f'}-1\big)}}
      \phi(f'')\cdot \phi(e'') \\
 &= \frac{1}{f}\sum_{f'f''=f} \phi(f'') \cdot \left(e,
   p^{f_0f'}-1\right)
\end{align*}
because we are adding $\phi(e'')$ for all $e''$ that divide both $e$
and $p^{f_0f'}-1$,
\begin{equation*}
= \frac{1}{f}\sum_{i=0}^{f-1} \left(e, p^{f_0(f',i)}-1\right)
\end{equation*}
because each divisor $f'$ of $f$ appears precisely $\phi(f/f') =
\phi(f'')$ times in the set of the $(f,i)$ for $i=0,\dots,f-1$.  This
is precisely the formula obtained in \cite[Remark 4.2,
pag. 27]{hou2004enumeration} when $p\nmid e$.
\end{remark}

With a computation similar to what done for Theorem
\ref{th1}, we obtain
\begin{theo}  
  The total number of isomorphism classes of extensions with degree
  $n$ of a field $K$ of absolute degree $n_0 = [K:\Q_p]$ and absolute
  inertia $f_0 = f(K/\Q_p)$ is
\begin{equation*}
  \kI(K,n) = \frac{1}{n} \sum_{
    \substack{0\leq i \leq t\\de'f'n^{(i)}=n}}
      e' \psi(k,p^{f_0f^{(i)}f'}-1)
      \cdot \Sigma_p(N', s')
      \cdot \Delta_p(N'+1,r,i),
\end{equation*}
where we are keeping the same notation as in Theorem \ref{th1},
$\psi(u,v)$ is defined in the \eqref{eq:2:2}, $p^t$ is the biggest
power of $p$ dividing $n$, and moreover throughout
the sum we have written $d=p^rk$ with $(k,p) = 1$.
\end{theo}

\section*{Acknowledgements}

We would like to thank Sara Checcoli, Roberto Dvornicich, Vincenzo
Mantova and Dajano Tossici for the many helpful discussions and
advices while writing down this note.

\providecommand{\bysame}{\leavevmode\hbox to3em{\hrulefill}\thinspace}
\providecommand{\MR}{\relax\ifhmode\unskip\space\fi MR }
\providecommand{\MRhref}[2]{%
  \href{http://www.ams.org/mathscinet-getitem?mr=#1}{#2}
}
\providecommand{\href}[2]{#2}

\end{document}